\documentclass[a4paper,11pt,reqno]{amsart}
\usepackage{amsmath,amssymb,amsthm} 
\usepackage{graphics} 
\usepackage{enumerate}

\usepackage[colorlinks,citecolor=blue]{hyperref}

\usepackage[capitalize]{cleveref}

\usepackage{epsfig}
\usepackage{color}
\usepackage[english]{babel}

\numberwithin{equation}{section}

\newtheorem{theorem}{Theorem}[section]
\newtheorem{lemma}[theorem]{Lemma}
\newtheorem{proposition}[theorem]{Proposition}

\newtheorem{corollary}[theorem]{Corollary}

\theoremstyle{definition}
\newtheorem{definition}[theorem]{Definition}
\newtheorem{remark}[theorem]{Remark}

\newcommand{\restr}{\mathop{\raisebox{-.127ex}{\reflectbox{\rotatebox[origin=br]{-90}{$\lnot$}}}}}

\newcommand{\R}{\mathbb{R}}
\newcommand{\N}{\mathbb{N}}
\newcommand{\C}{\mathbf{C}}

\newcommand{\eps}{\varepsilon}

\newcommand{\ds}{\displaystyle}

\newcommand{\be}{\begin{equation}}
\newcommand{\ee}{\end{equation}}

\DeclareMathOperator{\supp}{supp}

\newcommand\lt{\left}
\newcommand\rt{\right}

\def\les{\lesssim}
\def\ges{\gtrsim}

\newcommand{\cF}{\mathcal{F}}
\newcommand{\cB}{\mathcal{B}}
\newcommand{\cK}{\mathcal{K}}

\def \C{\mathbf{C}}

\def\EE{\mathbb{E}}
\def\PP{\mathbb{P}}
\def\diam{\operatorname{diam}}

\newcommand{\bip}{\operatorname{Bip}}
\newcommand{\cN}{\mathcal{N}}
\newcommand{\cM}{\mathcal{M}}
\newcommand{\cS}{\mathcal{S}}
\renewcommand{\c}{\mathsf{c}}

\newcommand{\bra}[1]{\left( #1 \right)}
\newcommand{\sqa}[1]{\left[ #1 \right]}
\newcommand{\cur}[1]{\left\{ #1 \right\}}

\newcommand{\abs}[1]{\left| #1 \right|}
\newcommand{\nor}[1]{\left\| #1 \right\|}
\def\fref{f}

\newcommand{\TSP}{\operatorname{TSP}}
\renewcommand{\C}{\mathcal{C}}

\newcommand{\pP}{\mathsf{P}}

\newcommand{\CPp}{\C_{\pP}^p}

\renewcommand{\TSP}{\mathsf{TSP}_\alpha}
\newcommand{\dist}{\mathsf{d}}

\newcommand{\Mp}{\mathsf{M}_\alpha}
\newcommand{\Mphalf}{\mathsf{M}_{1/2}}

\newcommand{\W}{\mathsf{W}_q}
\newcommand{\KY}{\mathsf{W}_\alpha}

\newcommand{\cT}{\mathcal{T}}
\newcommand{\cG}{\mathcal{G}}
\newcommand{\Lip}{\operatorname{Lip}}
\newcommand{\x}{{\bf x}}
\newcommand{\y}{{\bf y}}
\newcommand{\z}{{\bf z}}

\newcommand{\cR}{\mathcal{R}}
\newcommand{\cU}{\mathcal{U}}
\newcommand{\cV}{\mathcal{V}}

\newcommand{\cQ}{\mathcal{Q}}
\newcommand{\BV}{BV}

\newcommand{\cspan}{\mathsf{c}_{\operatorname{A2}}}
\newcommand{\cbdeg}{\mathsf{c}_{\operatorname{A3}}}
\newcommand{\cmerge}{\mathsf{c}_{\operatorname{A4}}}
\newcommand{\creg}{\mathsf{c}_{\operatorname{A5}}}

\usepackage[hyperref,doi,url=false,doi=false, isbn=false, giveninits=true, backref,style=numeric,maxbibnames=99, backend=biber]{biblatex}

\bibliography{OT.bib}

\title{On the concave one-dimensional random assignment problem and Young integration theory}
\author[M. Goldman]{Michael Goldman}
\address{M.G.:   CMAP, CNRS, \'Ecole polytechnique, Institut Polytechnique de Paris, 91120 Palaiseau,
France}
\email{michael.goldman@cnrs.fr}
\author[D. Trevisan]{Dario Trevisan}
\address{D.T.: Dipartimento di Matematica, Università degli Studi di Pisa, 56125 Pisa, Italy  }
\email{dario.trevisan@unipi.it}
\date{}
\subjclass[2010]{60D05, 90C05, 39B62, 60F25, 35J05}
\keywords{matching problem, optimal transport, geometric probability, Young integral}
\thanks{D.T. was
partially supported by the INdAM-GNAMPA project 2023 ``Temi di Analisi Armonica
Subellittica''.}
\newcounter{proof-step}

\begin{document}

\maketitle

\begin{center}
\emph{Dedicated to L.\ Ambrosio on its 60th birthday.}
\end{center}

\begin{abstract}
We investigate the one-dimensional random assignment problem in the concave case, i.e., the assignment cost is a concave power function, with exponent $0<p<1$, of the distance between $n$ source and $n$ target points, that are i.i.d.\ random variables with a common law on an interval. 
We prove that the limit of a suitable renormalization of the costs exists if the exponent $p$ is different than $1/2$. Our proof in the case $1/2<p<1$ makes use of a novel version of the Kantorovich optimal transport problem based on Young integration theory, where the difference between two measures is replaced by the weak derivative of a function with finite $q$-variation, which may be of independent interest. We also prove a similar result for the random bipartite Traveling Salesperson Problem.
\end{abstract}
\newcommand{\I}{I}
\newcommand{\cX}{\mathcal{X}}
\newcommand{\var}{\operatorname{var}}
\renewcommand{\PP}{\mathbb{P}}

\section{Introduction}

The assignment problem (or bipartite matching) is a classic optimization problem that arises in a variety of applications. The task is to  find an optimal correspondence  between two sets of objects $(x_i)_{i=1}^n$, $(y_i)_{i=1}^n$, such as workers and jobs, or producers and sellers of goods, such that the total cost optimized, e.g.\ minimized. 
 When the cost $c(x_i,y_j)$ between pairs of objects is a function of a distance, the solution naturally reflects some of the geometry of the underlying space. Already in the case of points on the line, it is well-known that a convex function favours monotone assignments, while a concave function, such as $c(x,y) = |x-y|^\alpha$ with $\alpha \in (0,1)$ yields a richer structure,  exhibiting a variety of hierarchies at different scales, with a compelling economic interpretation \cite{mccann1999exact}.

There is a rich literature studying random instances of combinatorial optimization problems in Euclidean spaces, stemming from the seminal paper \cite{beardwood1959shortest}. The main focus is on convergence results and concentration around the typical behavior for large instances of the problems, see e.g.\ the monographs \cite{yukich2006probability, steele1997probability}. When the problem is bipartite, i.e., formulated over  two random sets of points, such as the assignment problem, the classical methods encounter some limitations, due to local fluctuations of the number of samples in the two families \cite{BoutMar, BaBo}. 
The same fluctuations give rise to unexpected scaling behaviors of the costs, which tend to be asymptotically larger than their non-bipartite counterparts. This phenomenon was crucially observed in \cite{AKT84} for the assignment problem on the square, with the cost given by the  Euclidean distance. 

In recent years, progress has been made in the study of random Euclidean combinatorial problems, starting from the assignment problem, thanks to its link with Optimal Transport theory, which constitutes its natural linear programming relaxation and can be formulated for general source and target measures, not necessarily discrete ones. Thanks to the rich and much explored structure of optimal transport problems and their solutions, in particular for  absolutely continuous densities, several results have been obtained in the study of the random Euclidean bipartite matching problem \cite{CaLuPaSi14, caracciolo2015scaling, AmGlau, AmStTr16, AGT19, benedetto2021random, ambrosio2022quadratic, goldman2021quantitative, goldman2022fluctuation, } but also for other random bipartite combinatorial optimization problems, e.g.\ \cite{capelli2018exact, correddu2021minimum, goldman2022optimal}.

Aim of this paper is to provide a description of the asymptotic behaviour of the cost functional associated to the assignment problem over two sets of random i.i.d.\ points $(X_i)_{i=1}^n$, $(Y_i)_{i=1}^n \subseteq \R$, with respect to a concave power cost. We introduce the optimal assignment cost functional
\[ \Mp( (X_i)_{i=1}^n, (Y_i)_{i=1}^n) = \min_{\sigma \in\cS_n} \sum_{i=1}^n | X_i -  Y_{\sigma(i)} |^\alpha,\]
where $\cS_n$ denotes the set of permutations over $n$ elements and the exponent $\alpha \in (0,1)$ is fixed.  Assuming that the common law $\mu$ is supported on a bounded interval $\I$, the heuristics is that the points are equally spread over $\I$, hence at distance roughly $1/n$ from each other. Thus, we expect that $\Mp((X_i)_{i=1}^n, (Y_i)_{i=1}^n)$ should approximately behave like $n^{1-\alpha}$. This is true in the range $0<\alpha<1/2$, and upper and lower bounds for the renormalized cost $$ n^{\alpha-1}\Mp((X_i)_{i=1}^n, (Y_i)_{i=1}^n)$$ are established e.g.\ in \cite[Theorem 2]{BaBo}. However, to our knowledge, existence of the limit is proved only when the points are uniformly distributed over a set with positive Lebesgue measure (again e.g.\ in \cite[Theorem 2]{BaBo}). Of course, a similar result is putatively assumed to hold also for non-uniform distributions, and in this work we precisely settle such conjecture. 

In the range $1/2<\alpha<1$, local fluctuations in the number of points become dominant and one obtains an asymptotic rate of the order $n^{1/2}$ (actually, the same rate holds also for $\alpha\ge 1$). Such ``phase transition'' was investigated in e.g.\ \cite{caracciolo2020dyck, bobkov2020transport}, where upper bounds have been rigorously established. In this work we are able to give a precise description of the limit also in this case. The main idea in the proof is that after re-scaling by the order $n^{1/2}$, the matching problem converges to a suitable version of an optimal transport problem, where the two source and target measure are now replaced by a Brownian bridge, because of the Central Limit Theorem. Elaborating upon this idea, we propose an extension of the Kantorovich problem by relying upon Young's integration theory \cite{young1936inequality},  which yields in our case a robust description of the assignment problem in the limit $n \to \infty$, but we believe may be of independent interest.

The case $\alpha=1/2$ cannot be settled by our arguments, hence existence of the limit remains an open question. However, we complement the upper bounds from \cite{bobkov2020transport} with an asymptotic lower bound using a standard space-filling curve argument, showing indeed that the correct asymptotic rate is of the order $\sqrt{n \log n}$.

\subsection{Main result}

With the notation introduced above, we are in a position to state our main result concerning the asymptotic behaviour of the random assignment problem with respect to a concave power of the distance, on the real line.

\begin{theorem}\label{thm:main}
Let $(X_i)_{i=1}^\infty$, $(Y_i)_{i=1}^\infty \subseteq \R$ be i.i.d.\ random variables with common law $\mu$. Denote with $f$ the absolutely continuous part of $\mu$ (with respect to Lebesgue measure) and $F(t) = \mu((-\infty, t])$ the cumulative distribution of $\mu$.
\begin{enumerate}
\item If $\alpha \in (1/2, 1)$ and $\mu$ is supported on a bounded interval, then convergence in law holds:
\begin{equation}\label{eq:mp-limit-alpha-large}  \lim_{n \to \infty} n^{-1/2} \Mp( (X_i)_{i=1}^n, (Y_i)_{i=1}^n ) \to  \| \sqrt{2} B \circ F \|_{\KY},\end{equation}
where $(B(t))_{t \in [0,1]}$ denotes a standard Brownian bridge process. 
\item If $\alpha\in (0,1/2)$ and $\int_{\R} |x|^{\beta} d \mu < \infty$ for some $\beta >4 \alpha/(1-2 \alpha)$, then complete convergence  holds:
\[ \lim_{n \to \infty} n^{\alpha-1} \Mp( (X_i)_{i=1}^n, (Y_i)_{i=1}^n ) = c(\Mp) \int_\I f^{1-\alpha}(t) dt,\]
where $c(\Mp) \in (0,\infty)$ is a constant depending on $\alpha$ only. 
\end{enumerate}
In both cases, convergence holds also in expectation.

\end{theorem}

%

We also have an analogue theorem for the Travelling Salesperson Problem (TSP), which reads almost the same as \cref{thm:main}, with the cost $\TSP$ (defined in  \ref{sec:tsp}) in place of $\Mp$, and an additional factor $2$ in the right hand side of \eqref{eq:mp-limit-alpha-large}, and a different constant  $c(\TSP)$ instead of $c(\Mp)$. See \cref{thm:main-tsp} for a precise statement.

\subsection{Comments on the proof technique} As already mentioned, one of the main novelties is the formulation of a suitable Kantorovich-Young problem. In brief, for every function $g: \I \to \R$ over a bounded interval $\I$, with finite $q$-variation, and $\alpha \in (0,1)$ such that $\alpha+1/q>1$, we define
\[  \|g \|_{\KY} = \sup \cur{ \int_\I f d g \, : \,  \| f \|_{C^\alpha} \le 1},\]
where the integral is understood in the sense of Young.  Such a Kantorovich-Young problem  recovers the classical optimal transport when applied to functions $g$ with bounded variation (corresponding to  corresponds to the case of transporting two measures).  Indeed, in \cref{sec:ot} we investigate some basic properties of this problem, some of which are not strictly necessary in the proof of \cref{thm:main}. 

In the random case, we have that $g(t) = B(t)$ is a Brownian bridge, which has finite $q$-variation only if $q >2$, hence the Kantorovich-Young problem will be meaningful only if $\alpha>1/2$. This may also provide a explanation of the occurrence of the ``phase transition'' at $\alpha=1/2$, which in stochastic analysis motivates the introduction of Rough Paths theory \cite{friz2010multidimensional, friz2020course}.

Thus, for $\alpha<1/2$ we need to argue in a different way, and we do so by exploiting the so-called boundary functional associated to the assignment problem, where one is allowed to assign arbitrary points from/to the boundary of a given interval (for simplicity, we work in the case $\I = [0,1]$). The use of such functional is standard in the  theory of (non-bipartite) random combinatorial  optimization problems \cite{steele1997probability, yukich2006probability}, and is also advocated in \cite{BaBo}. It is however a widely open question to determine whether the asymptotic cost of the boundary functional coincides with that of the original assignment problem. In the present case of a concave cost on the interval, we are able to answer affirmatively to this question, by exploiting the specific structure of optimal assignments (in particular, we use the no-crossing property).

\subsection{Further questions}
Our result settles several questions  about existence of the limit for the renormalized costs of random bipartite assignment problem in the concave case on the line. Let us mention some problems that seem worth exploring.
\begin{enumerate}
\item For $1/2<\alpha <1$, \cref{thm:main} is limited to laws with bounded support. This is because we develop a minimal theory for the Kantorovich-Young problem, valid only in the case of a bounded interval. As with the classical transport problem, to address the case of unbounded intervals, some growth condition should be imposed.  In turn, this condition should then be verified in the convergence of the empirical process towards the Brownian bridge, hence extending the results from \cite{huang2001speed} which we employ in our argument. 
\item  A natural question is what happens in the $\alpha = 1/2$ case. It is known \cite[Corollary 3.2]{bobkov2020transport} that for i.i.d.\ points $(X_i)_{i=1}^\infty$, $(Y_i)_{i=1}^\infty \subseteq \I$ with common law $\mu$,
\begin{equation} \label{eq:bobkov-ledoux} \limsup_{n \to \infty} \EE\sqa{ \Mphalf( (X_i)_{i=1}^n, (Y_i)_{i=1}^n ) }/ \sqrt{ n \log n}  < \infty,\end{equation}
hence a natural conjecture would be that
\[ \lim_{n \to \infty} \EE\sqa{ \Mphalf( (X_i)_{i=1}^n, (Y_i)_{i=1}^n ) }/ \sqrt{ n \log n}  \]
always exists. In \cref{rem:alpha-1-2} below, we prove that if $\mu$ is  uniform on $\I$,  then
\begin{equation}\label{eq:lower-bound-half} \liminf_{n \to \infty}\EE\sqa{ \Mphalf( (X_i)_{i=1}^n, (Y_i)_{i=1}^n ) }/ \sqrt{ n \log n}   >0.\end{equation}
\item Our method is clearly not limited to the assignment problem, as the application to the TSP shows. We conjecture that the connected bipartite $\kappa$-factor problem, see  \cite{BaBo,  goldman2022optimal}, should allow for a similar analysis (the case $\kappa=2$ yields the TSP).

\end{enumerate}

\subsection{Structure of the paper} In \cref{sec:notation} we introduce the notation and collect some known facts on H\"older functions, $q$-variation,  optimal transport theory and the convergence of empirical processes towards Brownian bridges. In \cref{sec:assign}, we recall some simple properties of the assignment problem and the TSP in the concave case. We also establish a fundamental bound, \cref{lem:dirichlet-neumann-deterministic}, relating the boundary functional to the original problem. In \cref{sec:ot}, we introduce the Kantorovich-Young problem, and provide a duality result (\cref{prop:duality}) in terms of a suitable minimization of a transport cost over a set of couplings. Finally, in \cref{sec:main}, we  prove \cref{thm:main} and its variant for the TSP, \cref{thm:main-tsp}.

\subsection{Acknowledgements} D.T.\ thanks S.\ Caracciolo, A.\ Sportiello and M.\ D'Achille for valuable discussions about the statistical physics perspective on the problem.

\section{Notation and basic facts}\label{sec:notation}


We write throughout $\I$ for an interval $\I = [a,b] \subseteq \R$ (not necessarily bounded) with length $|I|$. To avoid measure-theoretic issues, we tacitly assume that all the functions $f: \I \to \R$ are right-continuous.

\subsection{H\"older functions}
Given a function $f: \I \to \R$, we define its H\"older semi-norm of exponent $\alpha \in (0,1]$, as
\[ \sqa{ f }_{C^\alpha} = \sup_{s\neq t \in \I} \frac{ \abs{ f(t) - f(s)}}{|t-s|^\alpha} \in [0, \infty].\]
Notice that $\sqa{ f }_{C^1}$ denotes the Lipschitz constant of $f$ (and not the usual $C^1$ norm). For $\alpha = 0$, we set $\sqa{f}_{C^0} = \sup_{s, t \in \I} \abs{f(t)-f(s)}$, the oscillation of $f$. To turns these quantities into norms, we define 
\[ \|f\|_{C^{\alpha}} = |f(c)| + [f]_{C^\alpha},\]
for a chosen $c \in \I$ (the precise choice is not relevant). We write as usual $C^{\alpha}(\I)$ for the Banach space of functions $f$ with finite norm $\|f\|_{C^\alpha} < \infty$. 
For $0<\beta<\alpha\le 1$, the inclusion $C^{\alpha}(\I) \subseteq C^\beta( \I)$ holds and, by Ascoli-Arzelà theorem, if $\I$ is bounded, the inclusion is also compact, i.e., for any bounded sequence $(f_n)_{n} \subseteq C^\beta(\I)$, one can extract converging a converging subsequence $(f_{n_k})_k$ towards some $f$, i.e., 
\[ \lim_{k \to \infty} \nor{ f_{n_k} - f}_{C^\alpha} = 0.\]
In particular, one also has pointwise convergence and $f \in C^{\beta}(\I)$ with
\[ [f]_{C^\beta} \le \liminf_{k\to \infty} [f_{n_k}]_{C^\beta}.\]

\subsection{$p$-variation}
For $p \in [1, \infty)$,  we define the $p$-variation semi-norm of $f : I \to \R$ as
\[\sqa{ f}_{p-\var} = \sup \cur{ \bra{ \sum_{i=1}^m |f(t_i)-f(t_{i-1}) |^p}^{1/p} \, : \cur{t_i}_{i=0}^m \subseteq I,  \quad  t_0 < t_1 < \ldots < t_m}.\]
We may also set $\sqa{f}_{\infty-\var} = \sqa{f}_{C^0}$, so that, for any $\alpha \in [0,1]$, we have the inequality
\begin{equation}\label{eq:f-hol-var-bound} [ f ]_{1/\alpha -\var} \le |\I|^{\alpha} [ f]_{C^\alpha}.\end{equation}
Moreover,
\begin{equation}\label{eq:oscillation-variation} [f]_{C^0} \le [f]_{p-\var}\end{equation}
for every $p \ge 1$. The $p$-variation decreases with respect to composition with increasing functions: if $J \subseteq \R$ is an interval and $j: J \to I$ is increasing (not necessarily continuous), then
\begin{equation}\label{eq:p-var-decreases} [ f \circ j]_{p-\var} \le [f]_{p-\var}.\end{equation}
When $p=1$, the $1$-variation $[f]_{1-\var}$ is simply called the total variation of $f$, and functions with finite $1$-variation may be also called functions with bounded variation ($\BV(\I)$).  Given two finite Borel measures $\mu^+$, $\mu^{-}$ over $\I$, the function
\[ f(t) := \mu^+(  (-\infty, t] \cap I) -  \mu^{-}( (-\infty, t] \cap I)\]
has finite total variation and it is well-known that, up to an additive a constant, (and choosing an a.e.\ representative to ensure right-continuity) any $f \in \BV(\I)$ can be represented as above. 
%
%
%
%


 \subsection{Young integration}
Given $f$, $g: I \to \R$, we say that the Riemann-Stieltjes integral $\int_I f d g$ is well defined if  the following limit exists
\[ \lim \sum_{i=1}^m f(t_{i-1}) (g(t_{i}) - g(t_{i-1})) = \int_I f dg\]
along any sequence of partitions $\{t_i\}_{i=0}^m \subseteq I$,  such that its mesh $\sup_{i=1, \ldots, m} | t_i - t_{i-1} |$ is infinitesimal as $m \to \infty$, and the limit does not depend on the chosen partitions. The (Lebesgue-)Riemmann-Stieljes theory of integration ensures that $\int_I f dg$ exists if both $[f]_{C^0}$ and $[g]_{1-\var}$ are finite. L.C.\ Young \cite{young1936inequality} established the following result. 

\begin{theorem}[Young]\label{thm:young}
Let $p$, $q \ge 1$ be such that $1/p+ 1/q >1$. Then, the  Riemann-Stieltjes integral $\int_I f dg$ exists for every pair of functions $f$, $g: I \to \R$ such that $f$, $g$ have no common points of discontinuity and both $\sqa{f}_{p-\var}$ and $\sqa{g}_{q-\var}$ are finite. Moreover, if $I = [a,b]$, it holds
\begin{equation}\label{eq:young-bound} \abs{ \int_I f  d g   - f(a)(g(b) -g(a))} \le C(p,q)   \sqa{f}_{p-\var} \sqa{g}_{q-\var},\end{equation}
where $C(p,q)\in (0,\infty)$ is a constant depending on $p$ and $q$ only (not even upon $\I$).
\end{theorem}

We are going to use the above construction in the special case of $f$ being H\"older continuous of exponent $\alpha = 1/p >1-1/q$, in particular continuous and of finite $p$-variation by \eqref{eq:f-hol-var-bound}, and $g$ with $g(a) = g(b)$ (if $I = [a,b]$). In such a case,  the Riemann-Stieltjes integral $\int_I f d g$ exists and, combining \eqref{eq:f-hol-var-bound} with  \eqref{eq:young-bound}, we have the inequality
\begin{equation}\label{eq:young-bound} \abs{ \int_I f  d g } \le C(p,q) | \I |^\alpha \| f\|_{C^{\alpha}} \nor{g}_{q-\var}.\end{equation}

\subsection{Wasserstein distance}
Given $q>0$ and positive Borel measures $\mu$, $\lambda$ on a Polish space $(\mathcal{X},\dist)$ with $\mu(\mathcal{X}) = \lambda(\mathcal{X}) \in (0, \infty)$ and finite $q$-th moments, i.e.,
\[ \int_{\cX} \dist(x,x_0)^q \mu(d x) + \int_{\cX}\dist(x,x_0)^q \lambda(d x)<\infty\]
for some (hence all) $x_0 \in \cX$, the optimal transport cost of order $q$ between $\mu$ and $\lambda$ is defined as the quantity
\begin{equation}\label{eq:wkz} \inf_{\pi\in\Gamma(\mu,\lambda)} \int_{\cX\times\cX} \dist(x,y)^q  \pi(d x,d y),\end{equation}
where $\Gamma(\mu, \lambda)$ is the set of couplings between $\mu$ and $\lambda$, i.e., finite Borel measures $\pi$ on the product $\cX \times \cX$ such that their marginals are respectively $\mu$ and $\lambda$. A simple compactness argument yields that the infimum in \eqref{eq:wkz} is always a minimum. Moreover, if $q \in (0,1]$, then \eqref{eq:wkz} actually defines a distance, while if $q \ge 1$ one needs to take the $q$-th root in \eqref{eq:wkz} to obtain a distance. With this convention, for every $q >0$, one defines the Wasserstein distance of order $q$ between $\mu$ and $\lambda$, which we denote by $\W(\mu, \lambda)$. Convergence with respect to $\W$ is easily characterized \cite{AGS}: for a given $q>0$, a sequence of measures $(\mu_n)_{n}$ with finite $q$-th moment converge towards $\mu$ with respect to the distance $\W$ if and only if $\mu_n \to \mu$ weakly, i.e., 
\[ \lim_{n \to \infty} \int_{\cX} f d \mu_n  = \int_\cX f d \mu \quad \text{for every $f: \cX \to \R$ bounded and continuous,}\]
and for some (hence all) $x_0 \in \cX$,
\[ \lim_{n \to \infty} \int_{\cX} \dist(x,x_0)^q  \mu_n(dx)   =  \int_{\cX} \dist(x,x_0)^q  \mu (dx).\] 

The optimal transport cost of any order $q$ admits a dual formulation. We are going to use only the case $q \in (0,1]$, which is particularly simple:
\begin{equation}\label{eq:W-duality} \W(\mu, \nu) = \sup_f\cur{ \int_\cX f \bra{ d\mu - d \nu} \, : \,  |f(x) - f(y)| \le \dist(x,y)^q \quad \forall x,y \in \cX}.\end{equation}

\subsection{Empirical processes} \label{sec:empirical} Let $(X_i)_{i=1}^\infty$ be i.i.d.\ real valued random variables with common law $\mu$ and cumulative distribution function
\[ F(t) = \mu( (-\infty, t] ), \quad \text{for $t \in \R$.}\]
For every $n \ge 1$, denote with $\mu_n = \frac 1 n \sum_{i=1}^n \delta_{X_i}$ the empirical process, whose cumulative distribution function is
\[ F_n(t) = \mu_n( (-\infty, t] ) = \frac 1 n \sum_{i=1}^n 1_{\cur{X_i \le t}}, \quad \text{for $t \in \R$.}\]
By the Central Limit Theorem,  $\sqrt{n}(F_n(t)-F(t))$ converges in law to a Gaussian process. Precisely, define a  Brownian bridge $(B(t))_{t \in [0,1]}$ as a continuous Gaussian process  $\EE\sqa{B(t)} =0$ and $\EE\sqa{ B(s)B(t) } =  \min\cur{s,t} - st$ for every $s, t \in [0,1]$. Then, one can approximate $\sqrt{n}(F_n(t)-F(t))$ with $B \circ F(t) $ in law. In \cite{huang2001speed}, extending a classical construction of Komlós, Major and Tusnády \cite{komlos1975approximation}, the convergence is made quantitative by providing, for every $q\in (2, \infty)$ and $n \ge 1$, a coupling between the process $\sqrt{n}(F_n-F)$ and a Brownian bridge $B_n$, such that
\begin{equation}\label{eq:dudley-huang}  \PP \bra{ [ \sqrt{n}(F_n-F) - B_n \circ F ]_{p-\var} >  x n^{-(q-2)/(2q)  }}  \le e^{-x/A(q)}, \quad \text{for every $x\ge q A(q)$,} \end{equation}
where $A(q) \in (0, \infty)$ denotes a constant depending on $q$ only.  This implies, for every $y \ge 1$, convergence in $L^p(\PP)$ of the $q$-variation:
\[ \EE\sqa{ [ \sqrt{n}(F_n-F) - B_n \circ F ]_{q-\var}^p }^{1/p} \le C(p,q) n^{-(q-2)/(2q)}.\]


%
Consider an additional family $(Y_i)_{i=1}^\infty$ of i.i.d.\ values with the same law, and independent of the family $(X_i)_{i=1}^\infty$. Denote with $\tilde{F}_n$ the associated empirical cumulative distribution function associated to the points $(Y_i)_{i=1}^n$, and $\tilde{B}_n$ a Brownian bridge coupled with $\sqrt{n}(\tilde{F}_n - F)$ such that the analogue of  \eqref{eq:dudley-huang} holds. Without loss of generality, we may assume that $B$ and $\tilde{B}$ are independent. Then,
\[ \operatorname{Law}(B_n \circ F - \tilde{B}_n \circ F)  = \operatorname{Law}(\sqrt{2} B_n \circ F)\]
and using the triangle inequality we obtain that, for every $p\ge 1$ and some constant $C(p,q)>0$,
\begin{equation}\label{eq:dudley-q} \EE\sqa{ [ \sqrt{n}(F_n-\tilde{F}_n) - \sqrt{2} B_n \circ F ]_{q-\var}^p  }^{1/p} \le C(p,q) n^{-(q-2)/(2q)} \end{equation}
for a suitable Brownian bridge process $B_n$.

\section{Assignment and other combinatorial optimization problems}\label{sec:assign}

\subsection{The assignment problem}
Given two families of points $(x_i)_{i=1}^n, (y_i)_{i=1}^n \subseteq \R$ the assignment (or bipartite matching) cost on the line, with exponent $\alpha \in (0,1]$, is defined as
\begin{equation}\label{eq:mp} \Mp( (x_i)_{i=1}^n, (y_i)_{i=1}^n ) = \min_{\sigma \in \cS_n} \sum_{i=1}^n | x_i - y_{\sigma(i)}|^\alpha,\end{equation}
where $\cS_n$ is the set of permutations over $\cur{1, \ldots, n}$. We call a minimizer $\sigma$ an optimal assignment between the points $(x_i)_{i=1}^n, (y_i)_{i=1}^n$.

By Birkhoff theorem, we have the identity
\begin{equation}\label{eq:birkhoff} \Mp( (x_i)_{i=1}^n, (y_i)_{i=1}^n ) =  \KY\bra{ \sum_{i=1}^n \delta_{x_i}, \sum_{i=1}^n \delta_{y_i}},\end{equation}
thus connecting the discrete combinatorial problem with a linear programming problem. 

Any optimal assignment $\sigma$ must satisfy the following condition (called monotonicity in the optimal transport literature)
\begin{equation}\label{eq:monotonicity-assignment} |x_i - y_{\sigma(i)}|^\alpha + |x_j - y_{\sigma(j)}|^\alpha \le |x_i - y_{\sigma(j)}|^\alpha + |x_j - y_{\sigma(i)}|^\alpha, \quad \text{for every $i$ and $j$,}\end{equation}
otherwise if the converse inequality hold for some $i$, $j$, one could get a strictly smaller cost by modifying $\sigma$ letting instead $i \mapsto \sigma(j)$, $j \mapsto \sigma(i)$.

Let us recall some simple consequences of \eqref{eq:monotonicity-assignment}. First, if $x_k = y_\ell$ for some $k$, $\ell$, then there exists an optimal assignment that maps $k$ into $\ell$. Indeed, let $\sigma$ be any optimal assignment and write \eqref{eq:monotonicity-assignment} with $i=k$ and $j = \sigma^{-1}(\ell)$, which becomes
\[  |y_{\sigma(k)} - x_k |^\alpha + | x_k - x_{ \sigma^{-1}(\ell) } |^\alpha \le |y_{\sigma(k)} - x_{ \sigma^{-1}(\ell) } |^\alpha,\]
i.e., a reverse triangle inequality. Therefore, we can modify $\sigma$ by mapping $k$ into $\ell$ and $\sigma^{-1}(\ell)$ into $\sigma(k)$ to obtain an assignment with the same cost, hence optimal.  

Next, we obtain the so-called no-crossing rule for optimal assignments. Given a permutation $\sigma \in \cS_n$ and $i,j \in \cur{1, \ldots, n}$,  the pair $(x_i, y_{\sigma(i)})$, $(x_j, y_{\sigma(j)})$ is said to be  crossing if the two  open intervals determined by the two pairs of points are neither disjoint, nor one includes the other.

\begin{lemma}[no-crossing rule]\label{lem:non-crossing-matching}
Let $\alpha \in (0,1)$. For any optimal assignment $\sigma$, no pair  $(x_i, y_{\sigma(i)})$, $(x_j, y_{\sigma(j)})$ is crossing, for $i,j \in \cur{1, \ldots, n}$. 
\end{lemma}

For $\alpha=1$, the thesis holds for at least one optimal assignment. Let us report the argument for later use.

\begin{proof}
Assume by contradiction that a pair $(x_i, y_{\sigma(i)})$, $(x_j, y_{\sigma(j)})$ is crossing. There are actually several cases to consider, but without loss of generality we focus only on the cases that the points are in the order $x_i < x_j < y_{\sigma(i)} < y_{\sigma(j)}$. Set $a = y_{\sigma(i)} - x_j>0$, $b = y_{\sigma(j)} - x_i>0$, and let $t \in (0,1)$ be such that
\[ y_{\sigma(i)} - x_i = (1-t) a + t b.\]
Since $a+b = y_{\sigma(i)} - x_i + y_{\sigma(j)} - x_j$, it also holds with the same $t$ that
\[ y_{\sigma(j)} - x_j = t a + (1-t)b.\]
By strict concavity, 
\[\begin{split} \abs{y_{\sigma(i)} - x_i}^{\alpha} + \abs{ y_{\sigma(j)} - x_j}^{\alpha}  & > (1-t)a^\alpha+ t b^\alpha + ta^\alpha +(1-t)  b^\alpha \\
 & = a^\alpha + b^\alpha\\
 & = \abs{y_{\sigma(j)} - x_i}^\alpha + \abs{y_{\sigma(i)} - x_j}^\alpha, \end{split}\]
 which contradicts \eqref{eq:monotonicity-assignment}.
\end{proof}


\subsection{Travelling Salesperson Problem}\label{sec:tsp}

Some of our arguments apply with minor modifications to other combinatorial optimization problems. To keep the exposition simple, let us limit ourselves to  the bipartite Traveling Salesperson Problem (TSP), where one searches for the cheapest cycle visiting two given families of points $(x_i)_{i=1}^n, (y_i)_{i=1}^n \subseteq \R$ and alternating between them. For an exponent $\alpha \in (0,1]$ the cost is defined as
\begin{equation}\label{eq:tsp} \TSP( (x_i)_{i=1}^n, (y_i)_{i=1}^n ) = \min_{\sigma, \tau \in \cS_n} \sum_{i=1}^n | x_{\sigma(i)} - y_{\tau(i)}|^\alpha + |y_{\tau(i)} - x_{\sigma(i+1)}|^\alpha,\end{equation}
where conventionally we let $\sigma(n+1) = \sigma(1)$. Instead of dealing with permutations, we may consider an abstract cycle $G$ on the complete bipartite graph $\bip_n$ over two copies of $\cur{1,\ldots, n}$. It is not difficult that there is a correspondence between such cycles and the parametrizations given by 
\[ (\sigma, \tau) \mapsto G = \cur{ (\sigma(i), \tau(i)), (\sigma(i+1), \tau(i) )}_{i=1}^n.\]
%
Reasoning in terms of graphs and cycles simplifies some arguments. For example, let us consider the analogue of \eqref{eq:monotonicity-assignment} for this problem. 
Let $(\sigma, \tau)$ be an optimizer in \eqref{eq:tsp} with associated cycle $G \subseteq \bip_n$. Then, for $i,j, k, \ell \in \cur{1, \ldots, n}$, such that $(i,k)$, $(j, \ell) \in G$ and $(i,\ell)$, $(j, k) \notin G$. Then,
\begin{equation}\label{eq:monotonicity-kfactor} | x_{i} - y_{k}|^\alpha + | x_{j} - y_{\ell}|^\alpha \le | x_{i} - y_{\ell}|^\alpha + | x_{j} - y_{k}|^\alpha.\end{equation}
Otherwise, one could obtain a strictly smaller cost by removing $(i,k)$, $(j, \ell)$ from $G$ and adding instead the edges $(i,\ell)$, $(j, k)$, an operation that still yields a cycle.


We deduce  from this inequality a \emph{no-too-many-crossings} rule. Given points $(x_i)_{i=1}^n$, $(y_i)_{i=1}^n \subseteq \R$, we say that a pair of edges $(i,k)$, $(j, \ell) \in \bip_n$ is crossing (or alternatively, that the edge $(j,\ell)$ crosses $(i,k)$) if the open intervals with extremes respectively $\cur{x_i, y_k}$, $\cur{x_j, y_\ell}$ are neither disjoint nor one includes the other. 

\begin{lemma}\label{lem:non-crossing-k-factor}
Let $\alpha \in (0,1)$,  $(x_i)_{i=1}^n$, $(y_i)_{i=1}^n \subseteq \R$ and $(\sigma, \tau)$ be an optimizer for \eqref{eq:tsp}, with associated cycle $G \subseteq \bip_n$. For every edge in $G$, there can be at most $2$ edges in $G$ that cross it.
\end{lemma}

\begin{proof}
Following the proof of \cref{lem:non-crossing-matching}, we see that for each edge $(j, \ell) \in G$ that crosses $(i,k)$, at least one edge among $(i,\ell)$, $(j, k)$ must belong to $G$, otherwise the argument (which uses only \eqref{eq:monotonicity-assignment} and strict concavity) would carry on using \eqref{eq:monotonicity-kfactor} instead. However, the degree of each vertex in $G$, in particular of $i$ and $k$, is exactly $2$, so this can happen at most $2$ times (because we take into account the  edge $(i,k)$ contributing to the degrees of $i$ and $k$).
\end{proof}

It will be useful to bound the cost of the TSP in terms of the assignment problem as follows.

\begin{lemma}\label{lem:capelli}
Given $\alpha \in (0,1)$,  for every $(x_i)_{i=1}^n$, $(y_i)_{i=1}^n \subseteq [0,1]$,  it holds
\begin{equation}\label{eq:boundk-fac-mp}  0 \le  \TSP((x_i)_{i=1}^n, (y_i)_{i=1}^n )  -    2 \Mp( (x_i)_{i=1}^n, (y_i)_{i=1}^n ) \le  1+ n^{1-\alpha}.\end{equation}
\end{lemma}

\begin{proof}
We argue as in \cite{capelli2018exact}, see also \cite{goldman2022optimal} and \cite{AGT19}. For the first inequality, simply notice that any cycle alternating between the $x$ and $y$'s induces a pair of assignments (not necessarily optimal), hence
\[ \TSP((x_i)_{i=1}^n, (y_i)_{i=1}^n ) \ge 2 \Mp( (x_i)_{i=1}^n, (y_i)_{i=1}^n ).\]
For the converse, assume without loss of generality that the $(x_i)_{i=1}^n$ are increasingly ordered. In \eqref{eq:tsp}, we let  $\sigma(i) = i$ the identity permutation, and $\tau \in \cS_n$ an optimal assignment between $(x_i)_{i=1}^n$ and $(y_i)_{i=1}^n$. By the triangle inequality, for every $i=1, \ldots, n$, we have
\[ | y_{\tau(i)} - x_{i+1} |^\alpha \le |y_{\tau(i)} - x_{i}|^\alpha + |x_{i} - x_{i+1}|^{\alpha},\]
(with the convention that $n+1 =1$), hence summing upon  $i$ gives
\[\begin{split} \TSP( (x_i)_{i=1}^n, (y_i)_{i=1}^n ) & \le \sum_{i=1}^n 2 |y_{\tau(i)} - x_{i}|^\alpha + |x_{i} - x_{i+1}|^{\alpha} \\
& = 2 \Mp((x_i)_{i=1}^n, (y_i)_{i=1}^n ) + 1 + \sum_{i=1}^{n-1} |x_{i}- x_{i+1}|^{\alpha} .\end{split}\]
To conclude, notice that, by H\"older inequality,
\[ \sum_{i=1}^{n-1} |x_{i}- x_{i+1}|^{\alpha} \le \bra{ \sum_{i=1}^{n-1} | x_i - x_{i+1}|}^{\alpha} n^{1-\alpha} = n^{1-\alpha},\]
hence the thesis.
\end{proof}

\subsection{Boundary functionals}
We recall the boundary functional associated to the assignment problem and the TSP from \cite{BaBo}. For simplicity, we consider only the case that $(x_i)_{i=1}^n, (y_i)_{i=1}^n \subseteq (0,1)$. We define, for the assignment problem,
\begin{equation}\label{eq:dirichelt-mp} \Mp^D( (x_i)_{i=1}^n, (y_i)_{i=1}^n) = \inf   \Mp( (x_i)_{i=1}^n \cup (\tilde{x}_j)_{j=1}^m, (y_i)_{i=1}^n  \cup (\tilde{y}_j)_{j=1}^m), \end{equation}
and for the TSP,
\begin{equation}\label{eq:dirichelt-kfac} \TSP^D( (x_i)_{i=1}^n, (y_i)_{i=1}^n) = \inf \TSP( (x_i)_{i=1}^n \cup (\tilde{x}_j)_{j=1}^m, (y_i)_{i=1}^n \cup  (\tilde{y}_j)_{j=1}^m),\end{equation}
where in both cases the infimum runs over all the possible families of \emph{boundary points} $(\tilde{x}_j)_{j=1}^m, (\tilde{y}_j)_{j=1}^m \subseteq \cur{0,1}$ with $m \in \mathbb{N}$ arbitrary. 

Clearly, the boundary functional is always smaller than the original cost functional (one may let $m=0$). In this section, we prove a converse inequality, up to some error.  The strategy is to prove that $m$ cannot be too large, and then use the following simple inequality.

\begin{lemma}\label{lem:remove-points}
Let $(x_i)_{i=1}^{n+m}$, $(y_i)_{i=1}^{n+m} \subseteq [0,1]$. Then,
\begin{equation}\label{eq:remove-m-assignment} \Mp( (x_i)_{i=1}^{n+m} , (y_i)_{i=1}^{n+m} ) \ge \Mp( (x_i)_{i=1}^n, (y_i)_{i=1}^n) - m,\end{equation}
and
\begin{equation}\label{eq:remove-m-factor} \TSP( (x_i)_{i=1}^{n+m} , (y_i)_{i=1}^{n+m} ) \ge \TSP( (x_i)_{i=1}^n, (y_i)_{i=1}^n) - 2 m.\end{equation}
\end{lemma}

\begin{proof}
It is sufficient to argue by induction, and prove the inequalities in the case $m=1$. 

To prove \eqref{eq:remove-m-assignment}, consider an optimal assignment $\sigma \in \cS_{n+1}$ for the left hand side.   If $\sigma(n+1) = n+1$, then $\sigma$ can be restricted on $\cur{1, \ldots, n}$ to obtain an assignment between the first $n$ points, and by dropping $|x_{n+1}-y_{\sigma(n+1)}|^\alpha$ from the left hand side, we obtain an upper bound for the right hand side. Otherwise, 
we define an assignment $\tilde{\sigma} \in \cS_n$  by letting $\tilde{\sigma}(\sigma^{-1}(n+1)) =  \sigma(n+1)$ and $\tilde{\sigma} = \sigma$ for all the remaining elements. We have therefore
\[ 
 \Mp( (x_i)_{i=1}^{n+1} , (y_i)_{i=1}^{n+1} )  
 \ge \Mp( (x_i)_{i=1}^{n} , (y_i)_{i=1}^{n} ) - |x_{\sigma^{-1}(n+1)} - y_{\sigma(n+1)}|^\alpha\] 
hence \eqref{eq:remove-m-assignment} with $m=1$, since $|x_{\sigma^{-1}(n+1)} - y_{\sigma(n+1)}|^\alpha \le 1$.

The argument leading to \eqref{eq:remove-m-factor} is similar but slightly more involved. Let $G \subseteq \bip_{n+1}$ be an optimizer for \eqref{eq:tsp}. Consider first the case that $x_{n+1}$, $y_{n+1}$ do not form an edge in $G$, i.e., $(n+1, n+1) \notin G$. Without loss of generality, i.e., up to relabelling the points, we can then assume that $x_{n+1}$ is connected to the points $y_1$, $y_2$ in $G$, and similarly $y_{n+1}$ to the points $x_1$, $x_2$. Since $G$ is a cycle, after removing $x_{n+1}$ and $y_{n+1}$,  
we have two connected components, $G_1$ and $G_2$. Up to relabeling the points, we may assume that $x_1$, $y_1 \in G_1$, and $x_2$, $y_2 \in G_2$. Thus, to define a cycle it is sufficient to add two pair of edges, one connecting $x_1$ to $y_2$ and one connecting $x_2$ to $y_1$. This yields \eqref{eq:remove-m-factor} with $m=1$ in the case that $x_{n+1}$ $y_{n+1}$ do not form an edge in $G$. If instead they are connected, it is sufficient to connect the two other points connected to them to form a cycle, which also  leads \eqref{eq:remove-m-factor} (actually without the factor $2$ in this case). 
\end{proof}




\begin{lemma}\label{lem:dirichlet-neumann-deterministic}
Let $\alpha \in (0,1)$, $n \ge 2$, $(x_i)_{i=1}^n$, $(y_i)_{i=1}^n \subseteq (0,1)$ be all distinct. Define, for $t \in [0,1]$, 
\[ F(t) = \sum_{i=1}^n \I_{\cur{x_i \le t}}, \quad \tilde{F}(t) = \sum_{i=1}^n \I_{\cur{y_i \le t}}.\]
Then, minimization in \eqref{eq:dirichelt-mp}  can be restricted to $m \le [ F - \tilde{F}]_{C^0}$, while minimization in \eqref{eq:dirichelt-kfac} can be restricted to $m \le [ F - \tilde{F}]_{C^0} + 1$.
\end{lemma}

\begin{proof}


Let us focus on the assignment problem first. Given $m$ pairs of boundary points $(\tilde{x}_{j})_{j=1}^m$, $(\tilde{y}_{j})_{j=1}^m \subseteq \cur{0,1}$, with $m > [ F - \tilde{F}]_{C^0}$, and an optimal assignment $\sigma \in \cS_{n+m}$, we argue that one can find a (not necessarily strictly) cheaper  assignment by removing a pair of points. Arguing recursively, this yields the first claim.

First, there is a trivial way to remove a pair of points: if $\tilde{x}_j = \tilde{y}_k$ for some $j,k=1, \ldots, m$, then for some optimal assignment (not necessarily $\sigma$) these two points are matched together, hence we can remove them obtaining a cheaper assignment with $m-1$ pairs. Hence, we can assume that $\tilde{x}_j \neq  \tilde{y}_k$ for every  $j, k=1, \ldots, m$. This implies that either $\tilde{x}_j=0$ (hence $\tilde{y}_j = 1$) for all $j$'s, or $\tilde{x}_j=1$ (hence $\tilde{y}_j=0$) for all $j$'s. For simplicity, let us consider only the first case (the other being symmetric). 

%
%
%
%

Let $j_0 \in \cur{1,\ldots, m}$ be such that $y_{\sigma(j_0)}$ is maximum among the values $\cur{y_{\sigma(j)}}_{j=1}^m$, i.e.\  $\tilde{x}_{j_0}$ is assigned to the rightmost position among all the $\tilde{x}_j$'s. We argue that $y_{\sigma(j_0)} = 1$, i.e.\ it is one among the $\tilde{y}_j$'s, hence by removing the pair $\tilde{x}_{j_0}$, $ y_{\sigma(j_0)}$, we obtain a family with $m-1$ pairs of boundary points, and (strictly) cheaper cost. Assume by contradiction that $y_{\sigma(j_0)} = t <1$. Then, the assumption $m> [ F - \tilde{F}]_{C^0}$, together the fact that the $x_i$'s and $y_i$'s for $i=1, \ldots, n$ are all distinct and $\tilde{y}_j = 1$ for $i=1, \ldots, m$, yields that, on the interval $[0, t)$, there are strictly more points from the family $(x_i)_{i=1}^n \cup (\tilde{x}_j)_{j=1}^m$ rather than the family $(y_i)_{i=1}^n \cup (\tilde{y}_j)_{j=1}^m$. Thus, there must be at least one $x_i \in [0, t)$ with $y_{\sigma(i)} \in (t,1]$. By construction, it cannot be $x_i = 0$, hence the pair $(x_i, y_{\sigma(i)})$ is crossing the pair $(\tilde{x}_{j_0}, y_{\sigma(j_0)}) = (0,t)$, which is a contradiction since $\sigma$ is assumed to be optimal. This concludes the argument for the assignment problem.

For the TSP,  we proceed along the same lines, with due modifications. Given $m$ pairs of boundary points $(\tilde{x}_{j})_{j=1}^m$, $(\tilde{y}_{j})_{j=1}^m \subseteq \cur{0,1}$, with $m > [ F - \tilde{F}]_{C^0} +1$, and an optimal cycle $G \subseteq \bip_{n+m}$, we find a cheaper cycle $G'$ with $m-1$ boundary points. 
Again, consider first the case that  for some $j,k=1, \ldots, m$, it holds $\tilde{x}_{j} = \tilde{y}_{k}$. Here, we argue as in the proof of \cref{lem:remove-points} with $\tilde{x}_j$ instead of $x_{n+1}$ and $\tilde{y}_k$ instead of $y_{n+1}$. The construction yields a cycle $G'$ which is cheaper, because we are substituting pairs of edges of the form $(\tilde{x}_j, u)$, $(v, \tilde{y}_k)$ for some $u$ among the $x$'s and $v$ among the $y$'s with a single edge $(u,v)$, leading to a lower cost by the triangle inequality:
\[ | u- \tilde{x}_j|^\alpha + |\tilde{y}_k - v|^\alpha \ge |u-v|^\alpha,\]
(since $\tilde{x}_j = \tilde{y}_k$). 

Thus, we can assume that $\tilde{x}_j \neq  \tilde{y}_k$ for every  $j, k=1, \ldots, m$ and for simplicity, we discuss only the  case  $\tilde{x}_j = 0$ for every $j=1,\ldots, m$. Let $y_{j_0}$ be maximum among the $y_i$'s (or the $\tilde{y}_j$'s) which share an edge in the cycle $G$ with some $\tilde{x}_j$, $j=1, \ldots, m$. For  simplicity assume that the edge is $(x_{j_0}, y_{j_0})$. We argue also in this case that $y_{j_0} = 1$, i.e.\ it is one of the $\tilde{y}_j$'s. Assume by contradiction that $y_{\sigma(j_0)} = t <1$. Then, the assumption $m> [ F - \tilde{F}]_{C^0} + 1$, together the fact that the $x_i$'s and $y_i$'s for $i=1, \ldots, n$ are all distinct and $\tilde{y}_j = 1$ for $i=1, \ldots, m$, yields that, on the interval $[0, t)$, there are at least $2$ more points from the family $(x_i)_{i=1}^n \cup (\tilde{x}_j)_{j=1}^m$ rather than the family $(y_i)_{i=1}^n \cup (\tilde{y}_j)_{j=1}^m$. Since each point in $G$ has degree exactly $2$, there must be at least two distinct edges in $G$ such that the $x$ point belongs to $[0, t)$ while its $y$ point is in $(t,1]$. By construction, the former cannot be $0$, hence the pairs are crossing the pair $(\tilde{x}_{j_0}, y_{\sigma(j_0)}) = (0,t)$, which is a contradiction since $G$ is assumed to be optimal. Hence, we have that $y_{j_0}=1$, i.e.\ it is one among the $\tilde{y}_j$'s. Let $\bar{x}$ denote the other point which shares an edge with $y_{j_0}$ (different than $x_{j_0}$) and $\bar{y}$ denote the other point which shares an edge with $x_{j_0}$ (different than $y_{j_0}$).  We now remove both $x_{j_0}$ and $y_{j_0}$ from the graph $G$, together with their associated edges, and add the edge $(\bar{x}, \bar{y})$, whose cost is anyway smaller than $1$. Notice that $(\bar{x}, \bar{y})$ cannot be already an edge in $G$ otherwise it would be a cycle of length $4$, which is not the case, since $n \ge 2$.\qedhere
\end{proof}

\section{A Kantorovich-Young problem}\label{sec:ot}


In view of identity \eqref{eq:W-duality}, taking into account Young's  \cref{thm:young}, one is lead to the following definition of a Kantovorich problem associated to a function $g: \I = [a,b] \to \R$ with $g(b) = g(a)$ and finite $q$-variation:
\begin{equation}\label{eq:g-w} \|g \|_{\KY} = \sup \cur{ \int_\I f d g \, : \,  \| f \|_{C^\alpha} \le 1},\end{equation}
provided that $\alpha+1/q>1$, so that the integral is well-defined. In particular, for some $C(\alpha, q) \in (0, \infty)$ it holds, by \eqref{eq:young-bound},
\[ \| g \|_{\KY} \le  C(\alpha, q) | \I |^\alpha \nor{g}_{q-\var}.\]
We have  immediately the following stability estimate:
\begin{equation}\label{eq:kantorovich-young-stability}  \abs{  \| g \|_{\KY} - \|\tilde{g}\|_{\KY}} \le \| g - \tilde{ g}\|_{\KY} \le C(\alpha,q ) | \I |^\alpha \nor{g - \tilde{ g}}_{q-\var}.
\end{equation}

It is quite natural to search for a primal problem, akin to  \eqref{eq:wkz}, which should be interpreted as an optimal transport problem. To this aim, we introduce a suitable notion of coupling associated to  $g: \I \to \R$ with $[g]_{q-\var} < \infty$. We say that a positive Borel measure $\pi$  on $\I \times \I$ (not necessarily finite) is a coupling for $g$ with finite $\alpha$-energy, and write $\pi \in \Gamma_\alpha(g)$, if
\begin{equation}\label{eq:pi-coupling} \int_{\I \times \I} |t-s|^\alpha  \pi (ds,dt) < \infty,\end{equation}
and
\begin{equation}\label{eq:divergence} \int_{\I \times \I} \bra{ f(t) -f(s) } \pi(ds,dt) =  \int_\I f dg ,\end{equation}
for every $f \in \C^\alpha(\I)$, where again the right hand side is well-defined by \cref{thm:young}.  In fact, one could completely avoid the use of Young's integration theory here, by noticing that \eqref{eq:pi-coupling} together with the validity of \eqref{eq:divergence} for $f \in C^1(\I)$ in the ``integrated-by-parts'' form
\[  \int_{\I \times \I} \bra{ f(t) -f(s) } \pi(ds,dt) = -  \int_\I g df, \]
can be used to extend the integration functional to a any $f \in C^\alpha(I)$,
\[ \int_\I f d g := \int_{\I \times \I} \bra{ f(t) -f(s) } \pi(ds,dt).\]

A further equivalent point of view on couplings is provided by defining the positive Borel measure $b(ds, dt) = |t-s|^\alpha  \pi (ds,dt)$, with has then finite total mass and rewrite \eqref{eq:divergence} as
\[  \int_{\I \times \I} \frac{ f(t) -f(s) }{|t-s|^\alpha} b(ds,dt) = \int_\I f d g.\]
which we should interpret as an analogue of $\operatorname{div} b = \mu - \lambda$  in the classical optimal transport problem (here $b$ plays the role of a flow).



The following duality result provides, as  expected, an identification between the two problems.

\begin{proposition}\label{prop:duality}
Let $\I=[a,b] \subseteq \R$, $q > 1$ and $g: \I \to \R$ be a function with finite $q$-variation and such that $g(a) = g(b)$. For every $\alpha \in (1-1/q,1]$, define $\| g\|_{\KY}$ as in \eqref{eq:g-w} where the integral $\int_\I f dg$ is in the sense of Young.  Then,  the supremum in \eqref{eq:g-w} is attained by some $f : \I \to \R$ with $[f]_{C^\alpha} = 1$ and
\begin{equation}\label{eq:kz-primal} \| g \|_{\W} = \min_{\pi \in \Gamma_\alpha (g)} \int_{\I \times \I} |t-s|^\alpha  \pi (ds,dt) < \infty.\end{equation}
\end{proposition}

We refer to a maximizer $f$ in \eqref{eq:g-w} as a Kantorovich-Young potential associated to $g$.

\begin{proof}
To show that \eqref{eq:g-w} is attained by some $f$, consider a maximizing sequence $f_n$, i.e.\ such that
\[ \lim_{n\to \infty} \int f_n d g = \| g\|_{\KY},\]
with $[f_n]_{C^\alpha} \le 1$ for every $n$. Up to adding a constant (which does not change the integrals $\int_\I f_n d g$ since $g(b) = g(a)$, we may assume that $f(a) = 0$. By compactness of the embedding $C^\alpha \subseteq C^\beta$, for $\beta<\alpha$, we can assume, up to extracting a subsequence, that $f_n \to f\in \C^{\beta}(\I)$ where convergence is in the sense of $C^\beta(\I)$ for some $\beta<\alpha$ such that $\beta+1/q>1$ (this is possible since $\alpha+1/q>1$). Moreover, by lower semicontinuity, $[f]_{C^\alpha} \le \liminf_n [f_n]_{C^\alpha}  =1$. Finally, using \eqref{eq:young-bound}, we deduce that
\[ \lim_{n \to \infty} \int (f_n -f) d g \le  \limsup_{n \to \infty} C(\beta, q) [f_n -f]_{C^\beta} [g]_{q-\var} = 0,\]
hence $\int _\I f d g  = \|g\|_{\KY}$. It must be $[f]_{C^\alpha} = 1$, otherwise letting $\tilde{f} = f/[f]_{C^{\alpha}}$ would lead to $\int _\I \tilde{f} d g  > \|g\|_{\KY}$, a contradiction.
 For the proof of \eqref{eq:kz-primal}, we use the Fenchel-Rockafellar duality theory as in \cite[Theorem 1.9]{villani2021topics}. Let $E = C( \I\times \I; \R)$ endowed with the uniform norm and continuous dual $E^* = \mathcal{M}( \I \times I)$ given by the signed Borel measures. We introduce the following convex functions defined on $E$,
\[ \Theta(u) = \begin{cases} 0 & \text{if $u(s,t) \ge - |t-s|^\alpha$, for every $s$, $t \in \I$,}\\
+ \infty & \text{otherwise,}\end{cases}\]
and
\[ \Xi(u) = \begin{cases} \int_\I f d g & \text{if $u(s,t) = f(t) - f(s)$ for some  $f \in C^\alpha(\I)$,}\\
+\infty & \text{otherwise.}\end{cases}\]
Notice that $\Xi$ is well defined since, if $f(t) - f(s) = \tilde{f}(t) - \tilde{f}(s)$, we obtain that $f(t) = \tilde{f}(t) + c$ for some constant $c \in \R$, and $\int_\I f d g = \int_\I \tilde{f} d g $ since $\int_\I c d g = c \bra{ g(b) - g(a)} = 0$. 
If we let $u_0(s,t) = 1$ for $s$, $t \in \I$, then $\Theta(u_0) = 0$ and $\Xi(u_0) = 0$, moreover $\Theta$ is continuous at $u_0$ (with respect to the uniform norm). By \cite[Theorem 1.9]{villani2021topics}, it follows that
\[ \inf_{u \in E} \sqa{ \Theta(u) + \Xi(u) } = \max_{\pi \in E^*} \sqa{ - \Theta^*(- \pi) - \Xi^*(\pi)},\]
where $\Theta^*$ and $\Xi^*$ denote the Legendre transforms of $\Theta$ and $\Xi$, respectively. Clearly,
\[ \int_{u \in E} \sqa{ \Theta(u) + \Xi(u) } = \inf\cur{ \int_I f dg  \, : \, f(t)-f(s) \ge -|t-s|^\alpha } =  - \| g \|_{\KY},\]
it is sufficient to recognize the dual problem as the opposite of the right hand side in \eqref{eq:kz-primal}. We easily see that
\[ \Theta^*( - \pi ) = \sup_{u \in E} \cur{ - \int_{\I \times \I} u d\pi -  \Theta(u)} =  \int_{\I \times \I} |t-s|^\alpha \pi(ds, dt),\]
while
\[ \Xi^*(\pi)  =\begin{cases} 0 & \text{if $\int_{\I \times \I} (f(t)-f(s)) \pi (ds, dt) = \int f d g$ for every $f \in C^\alpha(I)$,}\\
+ \infty & \text{otherwise.}\end{cases}\]
Therefore,
\[ \max_{\pi \in E^*} \sqa{ - \Theta^*(- \pi) - \Xi^*(\pi)} = \max_{\pi \in \Gamma_\alpha(g)} \cur{ -  \int_{\I \times \I} |t-s|^\alpha \pi(ds, dt)},\]
hence \eqref{eq:kz-primal}.
\end{proof}


\begin{remark}
The fact that $\Gamma_\alpha(g)$ is not empty is implicit in \eqref{eq:kz-primal}, however it may be interesting to notice that the construction of Young's integral amounts to providing a coupling with finite $\alpha$-energy. For simplicity, let us discuss only the case of a H\"older continuous $g \in C^{\beta}([0,1])$ with $\beta = 1/q$. Then, the argument simplifies as one can prove  (see e.g.\ \cite{feyel2006curvilinear}) that
\[ \int_0^1 f d g   = \sum_{n=1}^\infty \sum_{k=0}^{2^{n-1}-1} (f\bra{(2k)2^{-n}} - f\bra{(2k+1)2^{-n}} )(g\bra{(2k+2)2^{-n}}  - g\bra{(2k+1)2^{-n}}),\]
where we used that $g(0) = g(1) = 0$. Thus, by defining the measure
\[\begin{split} \pi & := \sum_{n=1}^\infty \sum_{k=0}^{2^{n-1}-1} (g\bra{(2k+2)2^{-n}}  - g\bra{(2k+1)2^{-n}})^+ \delta_{ ( (2k)2^{-n}, (2k+1)2^{-n})} \\
& \quad \quad  + (g\bra{(2k+2)2^{-n}}  - g\bra{(2k+1)2^{-n}})^- \delta_{ (  (2k+1)2^{-n}, (2k)2^{-n})} \end{split} \]
where $x^+ = \max\cur{x,0}$, $x^- =(-x)^+$ and $\delta_z$ denotes the Dirac measure at $z$, one obtains that $\pi \in \Gamma_\alpha(g)$.
\end{remark}

A relevant notion in optimal transport theory is cyclical monotonicity of optimal couplings, which we similarly recover here.

\begin{definition}
A set $\Gamma \subseteq  \R \times \R$ is said to be $\alpha$-monotone if, for every $((s_i, t_i) )_{i=1}^n \subseteq \Gamma$, and every permutation $\sigma$ over $\cur{1, \ldots, n}$,
\[  \sum_{i=1}^n |t_{i} - s_{i}|^\alpha  \le \sum_{i=1}^n |t_{i} - s_{\sigma(i)}|^\alpha .\]\end{definition}

Clearly, any subset $((s_i, t_i)) _{i=1}^n \subseteq \Gamma$ of an $\alpha$-monotone set $\Gamma$ provides an optimal assignment for the pair of points $(s_i)_{i=1}^n$, $(t_i)_{i=1}^n \subseteq \R$. In particular, it satisfies the no-crossing property of \cref{lem:non-crossing-matching}.

\begin{proposition}
If $\pi \in \Gamma_\alpha(g)$ is optimal, then $\supp \pi$ is $\alpha$-monotone. Viceversa, if $\pi \in \Gamma_\alpha(g)$ is concentrated on a $\alpha$-monotone set, then it is optimal. 
\end{proposition}

\begin{proof}
Let $f: \I \to\R$ be a Kantorovich-Young potential associated to $g$. Then, by \eqref{eq:kz-primal}
\[ \int_\I f d g = \int_{\I \times \I} |t-s|^\alpha \pi(ds, dt),\]
but also, by \eqref{eq:divergence},
\[\int_\I f d g = \int_{\I \times \I} \bra{ f(t) - f(s)} \pi(ds, dt).\]
Hence,
\[ \int_{\I \times \I} \bra{ |t-s|^\alpha - ( f(t) - f(s))  } \pi (ds, dt) =0,\]
but the integrand is non-negative for every $s$, $t \in \I$, since $[f]_{C^\alpha} \le 1$, hence $\pi$ is concentrated on the closed set
\[ \Gamma = \cur{ (s,t) \in \I \times \I \, : \, f(t) - f(s) = |t-s|^\alpha}.\]
One checks easily that $\Gamma$ is $\alpha$-monotone, since for every $((s_i, t_i))_{i=1}^n \subseteq \Gamma$ and $\sigma \in \cS_n$,
\[ \sum_{i=1}^n |t_i-s_i|^\alpha = \sum_{i=1}^n f(t_i) - f(s_i) = \sum_{i=1}^n f(t_{\sigma(i)}) - f(s_i) \le \sum_{i=1}^n | t_{\sigma(i)} - s_i|^\alpha.\]
For the converse, one uses the Smith and Knott's generalization of Rockafellar theorem \cite[theorem 6.14]{AGS} in order to construct $f: \I \to \R$ with $[f]_{C^\alpha} =1$ and such that
\[ \pi \bra{ \cur{ (s,t) \in \I \times \I \, : \, f(t) - f(s) < |t-s|^\alpha}} = 0.\]
Hence, 
\[ \int f dg = \int_{\I \times \I} \bra{ f(t) - f(s)} \pi(ds, dt) = \int_{\I \times \I} |t-s|^\alpha \pi(\ds, dt)\]
 and by \eqref{eq:kz-primal} we see that $f$ is Kantorovich-Young potential associated to $g$ and $\int_{\I \times \I} |t-s|^\alpha \pi(\ds, dt) = \| g \|_{\KY}$.
\end{proof}

\section{Proof of \cref{thm:main}}\label{sec:main}

We split the argument into the two cases: $0<\alpha<1/2$ and $1/2<\alpha<1$. In the latter case, we combine the results of \cref{sec:empirical}, in particular inequality \eqref{eq:dudley-q} with those of Section \ref{sec:ot}, in particular inequality \eqref{eq:kantorovich-young-stability}. For every $n$, there exists a Brownian bridge $B_n$ defined on a suitable probability space, together with the empirical cumulative distribution functions $F_n$, $\tilde{F}_n$ associated to $(X_i)_{i=1}^n$, $(Y_i)_{i=1}^n$, such that \eqref{eq:dudley-q} holds, for some $q>2$, such that $\alpha+1/q>1$. Since $B_n$ is $\beta$-H\"older continuous for every $\beta<1/2$, we have by \eqref{eq:p-var-decreases} and \eqref{eq:f-hol-var-bound} that $[ B_n \circ F]_{q -\var} \le [ B_n]_{q-\var} \le[B_n]_{C^{1/q}} <\infty$ $\mathbb{P}$-a.s.\ (and with finite moments of all orders). 

Recall that by \eqref{eq:birkhoff} we have the identity
\[ n^{-1/2} \Mp( (X_i)_{i=1}^n, (Y_i)_{i=1}^n) =  \| \sqrt{n} (F_n - \tilde{F}_n) \|_{\KY} .\]
By \eqref{eq:kantorovich-young-stability} and \eqref{eq:dudley-q}, we obtain
\[\begin{split} \EE& \sqa{ \abs{ \| \sqrt{n} (F_n - \tilde{F}_n) \|_{\KY} - \| \sqrt{2} B_n \circ F \|_{\KY} }^p}^{1/p}\\
 & \le C(\alpha, q) |\I|^{\alpha} \EE\sqa{ [ \sqrt{n} (F_n - \tilde{F}_n) - \sqrt{2} B_n \circ F]_{q-\var}^p }^{1/p} \le  C(\alpha, q) |\I|^{\alpha} n^{-(q-2)/(2q)},\end{split}\]
which yields the claimed convergence in law, as well as the convergence in expectation for $\alpha \in (1/2,1)$ (actually, we have convergence of the moments of any order). 
%
%
%
%


To settle the case $0<\alpha<1/2$, we rely  upon  the theory developed in \cite{BaBo}. Precisely, after \cite[Theorem 2, Theorem 19, Theorem 36]{BaBo} to obtain complete convergence and convergence in expectation, it is sufficient argue only in the case of  uniformly distributed points on the interval $[0,1]$, and identify the two limits
\begin{equation}\label{eq:dirichlet-neumann-alpha-small} \lim_{n \to \infty}n^{\alpha-1}  \EE\sqa{ \Mp^D( (X_i)_{i=1}^n, (Y_i)_{i=1}^n) } = \lim_{n \to \infty}n^{\alpha-1}  \EE\sqa{ \Mp( (X_i)_{i=1}^n, (Y_i)_{i=1}^n) }\end{equation}
(the fact that both limits exist is also proved in \cite{BaBo}). Since inequality $\le$ is trivially true, we only need to prove inequality $\ge$. Combining \cref{lem:remove-points} with \cref{lem:dirichlet-neumann-deterministic}, we reduce the problem to show that
\[ \lim_{n \to \infty} n^{\alpha-1} \EE\sqa{ [ F_n - \tilde{F}_n]_{C^0}  } = 0,  \]
where $F_n$, $\tilde{F}_n$ are the cumulative distribution functions associated to $(X_i)_{i=1}^n$, $(Y_i)_{i=1}^n$.  But this in turn is a simple consequence of \eqref{eq:dudley-q} with $p=1$,
\begin{equation}\label{eq:alpha-1-2-final}  \sqrt{n} \EE\sqa{ [ F_n - \tilde{F}_n]_{C^0} } \le  \EE\sqa{ \sqrt{n} [ F_n - \tilde{F}_n]_{q-\var} } \le \EE\sqa{ [B_n \circ F]_{q-\var}} + C(1,q) n^{-(q-2)/(2q)}.\end{equation}
 where the first inequality follows from \eqref{eq:oscillation-variation}. Multiplying both sides by $n^{\alpha-1/2}$ leads to the conclusion.

\begin{remark}\label{rem:alpha-1-2}
When $\alpha=1/2$, the argument above shows that
\[ \lim_{n \to \infty}   \frac{ \EE\sqa{ \Mphalf^D( (X_i)_{i=1}^n, (Y_i)_{i=1}^n) } - \EE\sqa{ \Mphalf( (X_i)_{i=1}^n, (Y_i)_{i=1}^n) }}{\sqrt{ n \log n } } =0,\]
(it is sufficient to divide both sides of \eqref{eq:alpha-1-2-final} by $\sqrt{ \log n}$). Thus, if the limits
\[ \lim_{n \to \infty}   \EE\sqa{ \Mphalf( (X_i)_{i=1}^n, (Y_i)_{i=1}^n) }/ {\sqrt{n \log n}}, \quad \lim_{n \to \infty}   \EE\sqa{ \Mphalf^D( (X_i)_{i=1}^n, (Y_i)_{i=1}^n) }/ {\sqrt{n \log n}},\]
exist, they coincide.  However, only the upper bound \eqref{eq:bobkov-ledoux} is known. Using a square-filling curve argument, we can prove
\begin{equation}\label{eq:lower-bound-peano} \liminf_{n \to \infty}   \EE\sqa{ \Mphalf^D( (X_i)_{i=1}^n, (Y_i)_{i=1}^n) }/ {\sqrt{n \log n}} >0,\end{equation}
which also implies the lower bound \eqref{eq:lower-bound-half}. Indeed, consider the Peano curve $\gamma: [0,1] \to [0,1]^2$, which is $1/2$-H\"older continuous and pushes the one-dimensional Lebesgue measure on $[0,1]$ into the two-dimensional Lebesgue measure on $[0,1]^2$, so that the points $(\gamma(X_i))_{i=1}^n$, $(\gamma(Y_i))_{i=1}^n$ are i.i.d.\ uniformly distributed on $[0,1]^2$. Since $\gamma(0) = (0,0)$ and $\gamma(1) = (1,1)$, the additional (random) boundary points $(\tilde{x}_j)_{j=1}^m$, $(\tilde{y})_{j=1}^m \subseteq \cur{0,1}$ for the problem $\Mphalf^D( ( X_i)_{i=1}^n, (Y_i)_{i=1}^n)$ are pushed via $\gamma$ to the boundary of $[0,1]^2$. Moreover, since $\gamma$ is $1/2$-H\"older continuous, $|\gamma(y) - \gamma(x)| \le [\gamma]_{C^{1/2}} |y-x|^{1/2}$ for every $s$, $t \in [0,1]$. It follows that
\[\mathsf{M}_1^D( ( \gamma(X_i))_{i=1}^n, (\gamma(Y_i) )_{i=1}^n) \le  [\gamma]_{C^{1/2}} \Mphalf^D( (X_i)_{i=1}^n, (Y_i )_{i=1}^n ),\]
where the left hand side denotes the boundary functional associated to the assignment problem on the square, with Euclidean cost. It is known \cite[Remark 3.3]{AGT19} (notice that there is a missing factor $\sqrt{n}$ there) that for i.i.d.\ uniformly distributed points one the square $(\tilde{X}_i)_{i=1}^n$, $(\tilde{Y}_i)_{i=1}^n \subseteq [0,1]^2$, it holds
\[ \liminf_{n \to \infty} \EE\sqa{ \mathsf{M}_1^D( ( \tilde X_i))_{i=1}^n, (\tilde Y_i) )_{i=1}^n)  } / {\sqrt{n \log n} }>0,\]
hence  \eqref{eq:lower-bound-peano}.

\end{remark}

We end this section stating the variant of our main result for the TSP.  For a proof, we combine straightforwardly \cref{thm:main} with  \cref{lem:capelli} in the case $\alpha>1/2$, while for $\alpha<1/2$, we follow the same argument as in the assignment problem, using \cref{lem:remove-points} and \cref{lem:dirichlet-neumann-deterministic} in the TSP case.

\begin{theorem}\label{thm:main-tsp}
Let $(X_i)_{i=1}^\infty$, $(Y_i)_{i=1}^\infty \subseteq \R$ be i.i.d.\ points with common law $\mu$. Denote with $f$ the absolutely continuous part of $\mu$ (with respect to Lebesgue measure) and $F(t) = \mu((-\infty, t])$ the cumulative distribution of $\mu$.
\begin{enumerate}
\item If $\alpha \in (1/2, 1)$ and $\mu$ is supported on a bounded interval, then convergence in law holds:
\begin{equation}\label{eq:mp-limit-alpha-large}  \lim_{n \to \infty} n^{-1/2} \TSP( (X_i)_{i=1}^n, (Y_i)_{i=1}^n ) = 2  \| \sqrt{2} B \circ F \|_{\KY},\end{equation}
where $(B(t))_{t \in [0,1]}$ denotes a Brownian bridge process. 
\item If $\alpha\in (0,1/2)$ and $\int_{\R} |x|^{\beta} d \mu < \infty$ for some $\beta >4 \alpha/(1-2 \alpha)$, then complete convergence  holds:
\[ \lim_{n \to \infty} n^{\alpha-1} \TSP( (X_i)_{i=1}^n, (Y_i)_{i=1}^n ) = c(\TSP) \int_\I f^{1-\alpha}(t) dt,\]
where $c(\TSP) \in (0,\infty)$ is a constant, depending on $\alpha$ only. 
\end{enumerate}
In both cases, convergence holds also in expectation.

\end{theorem}

\printbibliography

\end{document}